\documentclass[12pt]{article}
\usepackage[T1]{fontenc}
\usepackage[margin=1in]{geometry}
\usepackage{amsmath,amssymb,amsthm}
\usepackage{graphicx}
\usepackage{booktabs}
\usepackage{algorithm}
\usepackage{algpseudocode}

\usepackage{xcolor}

\newcommand{\bx}{\mathbf{x}}
\newcommand{\by}{\mathbf{y}}
\newcommand{\bz}{\mathbf{z}}
\newcommand{\bw}{\mathbf{w}}
\newcommand{\br}{\mathbf{r}}
\newcommand{\bxi}{\boldsymbol{\xi}}
\newcommand{\bfeta}{\boldsymbol{\eta}}
\newcommand{\vth}{\boldsymbol{\vartheta}}

\newcommand{\E}{\mathbb{E}}
\newcommand{\R}{\mathbb{R}}
\newcommand{\Rn}{\mathbb{R}^n}
\newcommand{\Rd}{\mathbb{R}^d}
\newcommand{\Lop}{\mathcal{L}}
\newcommand{\hatm}{\hat{m}}
\newcommand{\hatE}{\widehat{E}}

\theoremstyle{plain}
\newtheorem{proposition}{Proposition}
\newtheorem{remark}{Remark}

\title{\textbf{Weak Adversarial Neural Pushforward Method for the\\ McKean--Vlasov / Mean-Field Fokker--Planck Equation}}

\author{Andrew Qing He\thanks{Department of Mathematics, Southern Methodist University,
Dallas, TX, USA. \texttt{andrewho@smu.edu}}
\and
Wei Cai\thanks{Department of Mathematics, Southern Methodist University,
Dallas, TX, USA. \texttt{cai@smu.edu}}}

\date{March 23, 2026}

\begin{document}
\maketitle

\begin{abstract}
We extend the Weak Adversarial Neural Pushforward Method (WANPM) \cite{HeCai2025} to
the McKean--Vlasov mean-field Fokker--Planck equation, covering both the stationary and
time-dependent cases. The key observation is that the mean-field nonlinearity --- an
expectation under the solution distribution --- is naturally estimated by Monte Carlo
sampling from the pushforward network, requiring no change to the architecture and only
minor modifications to the training loop. For the quadratic (granular media) interaction
kernel, the interaction term reduces to the batch sample mean, eliminating secondary
sampling entirely. We also identify a dimension-dependent frequency initialization rule
for the adversarial test functions, necessary to avoid spurious minimizers. Numerical experiments on linear McKean--Vlasov benchmarks in 2,
5, 20, and 100 dimensions confirm accurate recovery of the exact Gaussian stationary
and transient distributions, with training times ranging from 27 seconds (2D) to 10
minutes (100D) on a single GPU.
\end{abstract}


\tableofcontents

\section{Introduction}

The standard Fokker--Planck equation (FPE) governs the time evolution of the probability
density $\rho(t,\bx)$ of a stochastic process whose drift and diffusion coefficients are
prescribed functions of $(t,\bx)$. A natural and physically important generalization arises
when these coefficients depend on the evolving distribution $\rho_t$ itself, yielding the
McKean--Vlasov or mean-field Fokker--Planck equation:
\begin{equation}
  \partial_t \rho(t,\bx)
  = -\nabla_{\bx} \cdot \bigl[b(\bx,\rho_t)\,\rho(t,\bx)\bigr]
    + \frac{\sigma^2}{2}\,\Delta\rho(t,\bx),
  \label{eq:MKV}
\end{equation}
with initial condition $\rho(0,\bx)=\rho_0(\bx)$. This class of equations appears in
interacting particle systems and granular media \cite{McCann1997}, Curie--Weiss spin
systems, stochastic flocking \cite{HaTadmor2008}, mathematical finance and mean-field
games \cite{CarmonaDelarue2018}, and the mean-field analysis of overparameterized
neural networks \cite{MeiMontanariNguyen2018}.

The nonlinearity of~\eqref{eq:MKV} --- the unknown $\rho$ appears both as the evolved
density and inside the convolution integral defining the drift --- makes it significantly
harder to solve than the standard FPE. Particle methods exploit the propagation of chaos
phenomenon \cite{Sznitman1991}: as $N\to\infty$, the empirical measure of $N$
interacting particles converges to the solution of~\eqref{eq:MKV}. However, particle
methods suffer from statistical noise, require a large number of particles for accuracy,
and do not naturally produce a functional representation of $\rho$ queryable at arbitrary
points.

Neural network methods for solving Fokker--Planck equations have been developed in
recent works \cite{XuFP2020,LiuNPFP2022,EHanLi2019,ZhaiDobsonLi2022}. In
\cite{HeCai2025}, the Weak Adversarial Neural Pushforward Method (WANPM) was
introduced, which learns a neural pushforward map $F_{\vth}:\Rd\to\Rn$ that transforms
samples from a simple base distribution into samples from the solution distribution,
trained via an adversarial weak formulation with plane-wave test functions. WANPM avoids
the need for explicit density representations, invertible architectures, or score
estimation, and scales to high-dimensional problems.

The present paper extends WANPM to the McKean--Vlasov setting. For the quadratic
(granular media) kernel $W(\bx-\by)=\frac{1}{2}\|\bx-\by\|^2$, the mean-field
nonlinearity reduces to a sample mean, requiring no secondary sampling. We develop both
the stationary and time-dependent formulations, analyze training subtleties that arise
from the self-referential structure of the equation, and demonstrate accuracy across
dimensions from 2 to 100.

The paper is organized as follows. Section~\ref{sec:eq} introduces the McKean--Vlasov
equation and the quadratic kernel simplification. Section~\ref{sec:stationary} develops
WANPM for the stationary problem. Section~\ref{sec:timedep} extends the method to the
time-dependent case, including a discussion of adversarial frequency initialization.
Section~\ref{sec:experiments} presents numerical experiments, and Section~\ref{sec:conclusion}
concludes.

\section{The McKean--Vlasov Equation}
\label{sec:eq}

\subsection{Derivation from an Interacting Particle System}

Consider $N$ particles with positions $X_i(t)\in\Rn$ obeying
\begin{equation}
  dX_i = b\!\left(X_i,\,\frac{1}{N}\sum_{j=1}^N \delta_{X_j(t)}\right)dt + \sigma\,dW_i,
  \quad i=1,\ldots,N,
  \label{eq:particles}
\end{equation}
where $W_i$ are independent standard Brownian motions and $\delta_{\by}$ denotes the
Dirac measure at $\by$. By the propagation of chaos phenomenon \cite{Sznitman1991},
as $N\to\infty$ the single-particle density satisfies the McKean--Vlasov equation~\eqref{eq:MKV}.

\subsection{The Aggregation--Diffusion Form}

The most widely studied instance takes the drift as the superposition of a confinement
force and a pairwise interaction:
\begin{equation}
  b(\bx,\rho_t) = -\nabla V(\bx) - \nabla(W*\rho_t)(\bx)
  = -\nabla V(\bx) - \int_{\Rn}\nabla W(\bx-\by)\,\rho(t,\by)\,d\by.
  \label{eq:drift}
\end{equation}
Substituting~\eqref{eq:drift} into~\eqref{eq:MKV} yields the aggregation--diffusion
equation
\begin{equation}
  \partial_t \rho
  = \nabla\cdot\bigl[(\nabla V)\rho\bigr]
  + \nabla\cdot\bigl[(\nabla W*\rho)\rho\bigr]
  + \frac{\sigma^2}{2}\Delta\rho.
  \label{eq:aggdiff}
\end{equation}
The unknown $\rho$ appears both as the evolved density and inside the convolution; this
self-referential structure is the central mathematical difficulty.

\subsection{Quadratic Interaction Kernel and Complexity Reduction}
\label{sec:quadratic}

We focus on the quadratic (granular media) kernel
$W(\bx-\by)=\frac{1}{2}\|\bx-\by\|^2$, for which a key simplification holds.

\begin{proposition}[\cite{McCann1997}]
\label{prop:quadratic}
Let $W(\bx-\by)=\frac{1}{2}\|\bx-\by\|^2$. Then
\begin{equation}
  \int_{\Rn}\nabla W(\bx-\by)\,\rho_t(\by)\,d\by
  = \bx - m(t),
  \qquad m(t)=\E_{\by\sim\rho_t}[\by].
  \label{eq:quadreduction}
\end{equation}
\end{proposition}
\begin{proof}
Direct computation: $\nabla_{\bx}W(\bx-\by)=\bx-\by$, so
$\int(\bx-\by)\rho_t(\by)\,d\by=\bx-m(t)$.
\end{proof}

With confinement $V(\bx)=\frac{\theta}{2}\|\bx\|^2$ and the quadratic kernel, the full
drift simplifies to
\begin{equation}
  b(\bx,\rho_t) = -\theta\bx - (\bx - m(t))
  = -(\theta+1)\bx + m(t) =: -\lambda\bx + m(t),
  \label{eq:lineardrift}
\end{equation}
where $\lambda:=\theta+1$ is the effective mean-reversion rate. The interaction integral
reduces to a single scalar $m(t)=\E[\bx]$, which is the sample mean of the primary
training batch during training --- no secondary sampling is needed.

\section{WANPM for the Stationary McKean--Vlasov Equation}
\label{sec:stationary}

\subsection{The Stationary Problem}

Setting $\partial_t\rho=0$ in~\eqref{eq:MKV} yields
\begin{equation}
  -\nabla\cdot\bigl[b(\bx,\rho)\,\rho(\bx)\bigr]
  + \frac{\sigma^2}{2}\Delta\rho(\bx) = 0,
  \quad \int_{\Rn}\rho(\bx)\,d\bx=1.
  \label{eq:stationary}
\end{equation}
This is a nonlinear eigenvalue problem: the operator acting on $\rho$ depends on $\rho$
itself through the convolution. For the quadratic kernel with $V(\bx)=\frac{\theta}{2}\|\bx\|^2$,
the drift~\eqref{eq:lineardrift} gives $b(\bx,\rho)=-\lambda\bx+m^*$ where
$m^*=\E_\rho[\bx]$. By the symmetry of $V$ and $W$, the unique stationary solution is
$\rho^*=\mathcal{N}(0,\frac{\sigma^2}{2\lambda}I)$ with $m^*=0$, giving
$b(\bx,\rho^*)=-\lambda\bx$.

\subsection{Weak Formulation}
\label{sec:weakstat}

To derive the weak form, multiply~\eqref{eq:stationary} by a smooth test function
$f(\bx)$ and integrate over $\Rn$. The right-hand side becomes
\begin{equation*}
  \int_{\Rn} f(\bx)\,\nabla\cdot\!\left[\rho(\bx)\,\nabla\!\left(
    V(\bx) + (W*\rho)(\bx) + \tfrac{\sigma^2}{2}\rho(\bx)
  \right)\right]d\bx.
\end{equation*}
One integration by parts moves the divergence onto $\nabla f$:
\begin{equation*}
  = -\int_{\Rn}\sum_{i=1}^n \frac{\partial f}{\partial x_i}
    \left[\rho(\bx)\,\frac{\partial}{\partial x_i}\!\left(
      V(\bx) + (W*\rho)(\bx) + \tfrac{\sigma^2}{2}\rho(\bx)
    \right)\right]d\bx.
\end{equation*}
Separating the nonlinear diffusion term $\frac{\sigma^2}{2}\rho$ from the remaining
drift and applying integration by parts once more to that term gives
\begin{align*}
  &= -\int_{\Rn}\rho(\bx)\sum_{i=1}^n\frac{\partial f}{\partial x_i}
      \frac{\partial}{\partial x_i}\!\Bigl(V(\bx)+(W*\rho)(\bx)\Bigr)\,d\bx
   + \frac{\sigma^2}{2}\int_{\Rn}\rho(\bx)\sum_{i=1}^n
      \frac{\partial^2 f}{\partial x_i^2}\,d\bx.
\end{align*}
Expanding the interaction term $(W*\rho)(\bx)=\int_{\Rn}W(\bx-\by)\rho(\by)\,d\by$
and collecting by role, the right-hand side decomposes as
\begin{equation}
  -E_V - E_W + E_D = 0,
  \label{eq:weakstat}
\end{equation}
where
\begin{align}
  E_V &= \int_{\Rn}\rho(\bx)\sum_{i=1}^n
    \frac{\partial f(\bx)}{\partial x_i}\frac{\partial V(\bx)}{\partial x_i}\,d\bx,
    \label{eq:EV_stat}\\
  E_W &= \int_{\Rn}\!\int_{\Rn}\rho(\bx)\rho(\by)\sum_{i=1}^n
    \frac{\partial f(\bx)}{\partial x_i}\frac{\partial W}{\partial x_i}(\bx-\by)\,d\by\,d\bx,
    \label{eq:EW_stat}\\
  E_D &= \frac{\sigma^2}{2}\int_{\Rn}\rho(\bx)\sum_{i=1}^n
    \frac{\partial^2 f(\bx)}{\partial x_i^2}\,d\bx.
    \label{eq:ED_stat}
\end{align}
Equation~\eqref{eq:weakstat} is the weak stationarity condition that the pushforward
network is trained to satisfy.

Following \cite{HeCai2025}, we choose plane-wave test functions
$f^{(k)}(\bx)=\sin(\bw^{(k)}\cdot\bx+b^{(k)})$, whose derivatives are all analytic:
\begin{equation}
  \nabla f^{(k)}(\bx) = \bw^{(k)}\cos(\bw^{(k)}\cdot\bx+b^{(k)}),
  \qquad
  \Delta f^{(k)}(\bx) = -\|\bw^{(k)}\|^2\sin(\bw^{(k)}\cdot\bx+b^{(k)}).
  \label{eq:planewavederiv}
\end{equation}
For the quadratic kernel, substituting~\eqref{eq:lineardrift} into $E_V+E_W$ and using
$\nabla W(\bx-\by)=\bx-\by$ gives the combined drift contribution
\begin{equation}
  (E_V+E_W)\big|_\text{quad}
  = \int_{\Rn}\rho(\bx)\,(-\lambda\bx+m^*)\cdot\bw^{(k)}\cos(\bw^{(k)}\cdot\bx+b^{(k)})\,d\bx,
  \label{eq:residual_stat}
\end{equation}
where $m^*=\E_\rho[\bx]$ is the only quantity depending nonlinearly on $\rho$.

\subsection{Neural Parametrization}

We parametrize the stationary distribution $\rho^*$ via a neural pushforward map
$F_{\vth}:\Rd\to\Rn$, where $d$ is the dimension of the base distribution (not
necessarily equal to $n$). If $\br\sim\pi_\mathrm{base}$, then $\bx=F_{\vth}(\br)$ is a
sample from the learned distribution $\rho_{\vth}$. The pushforward parametrization
naturally enforces the normalization constraint. In the stationary setting, $F_{\vth}$
takes only $\br$ as input.

The $K$ test functions have learnable parameters $\eta=\{\bw^{(k)},b^{(k)}\}_{k=1}^K$.

\subsection{Loss Function}

The training objective is the adversarial squared weak residual:
\begin{equation}
  \mathcal{L}[\vth,\eta] = \frac{1}{K}\sum_{k=1}^K
  \left[\frac{1}{M}\sum_{m=1}^M
    \Bigl(-E_V^{(k,m)} - E_W^{(k,m)} + E_D^{(k,m)}\Bigr)
  \right]^2,
  \label{eq:loss_stat}
\end{equation}
where $\{\bx^{(m)}=F_{\vth}(\br^{(m)})\}_{m=1}^M$ is a batch of pushed samples,
the terms $E_V^{(k,m)}, E_D^{(k,m)}$ are the per-sample contributions
to~\eqref{eq:EV_stat}--\eqref{eq:ED_stat} evaluated via~\eqref{eq:planewavederiv},
and the mean $m^*$ is approximated by the batch mean
\begin{equation}
  \hatm = \frac{1}{M}\sum_{m=1}^M F_{\vth}(\br^{(m)}),
  \quad \br^{(m)}\sim\pi_\mathrm{base},
  \label{eq:batchmean}
\end{equation}
computed from the same batch as the residual so that $\partial\hatm/\partial\vth$
contributes to the total gradient, enforcing the self-consistency constraint
$m^*=\E_{\rho_{\vth}}[\bx]$ through the optimization dynamics.

\subsection{Training Algorithm}

The min-max optimization problem is
\begin{equation}
  \min_{\vth}\max_{\eta}\,\mathcal{L}[\vth,\eta],
  \label{eq:minmax}
\end{equation}
solved by alternating gradient descent on $\vth$ and gradient ascent on $\eta$.
Algorithm~\ref{alg:stationary} summarizes the training loop.

\begin{algorithm}[h]
\caption{WANPM for the Stationary McKean--Vlasov Equation}
\label{alg:stationary}
\begin{algorithmic}[1]
\Require Pushforward net $F_{\vth}$, test functions $\{f^{(k)}\}_{k=1}^K$ with
$\bw^{(k)}$ initialized at scale $\sigma_w$ (see Section~\ref{sec:freqinit}),
batch sizes $M$, $M_W$, learning rates $\eta_\mathrm{gen}$, $\eta_\mathrm{test}$.
\For{each training epoch}
  \State Sample $\{\br^{(m)}\}_{m=1}^M\sim\pi_\mathrm{base}$
  \State Compute pushed samples $\bx^{(m)}=F_{\vth}(\br^{(m)})$
  \State Compute batch mean $\hatm=\frac{1}{M}\sum_m\bx^{(m)}$
    \quad {\small(gradient flows through $\hatm$)}
  \State Evaluate $\Lop[\rho_{\vth}]\,f^{(k)}(\bx^{(m)})$ using~\eqref{eq:residual_stat}
    with $m^*\leftarrow\hatm$
  \State Compute residuals $R^{(k)}=\frac{1}{M}\sum_m\Lop[\rho_{\vth}]\,f^{(k)}(\bx^{(m)})$
  \State Compute loss $\mathcal{L}=\frac{1}{K}\sum_k(R^{(k)})^2$
  \State Generator step: $\vth\leftarrow\vth-\eta_\mathrm{gen}\nabla_{\vth}\mathcal{L}$
  \If{adversary update step}
    \State $\eta\leftarrow\eta+\eta_\mathrm{test}\nabla_{\eta}\mathcal{L}$
  \EndIf
\EndFor
\end{algorithmic}
\end{algorithm}

\section{Extension to the Time-Dependent Problem}
\label{sec:timedep}

\subsection{Weak Formulation}
\label{sec:weaktime}

To derive the weak form for the time-dependent problem, multiply~\eqref{eq:MKV} by a
smooth test function $\psi(t,\bx)$ and integrate over $\Rn$. The spatial integral of
the right-hand side proceeds identically to the stationary derivation of
Section~\ref{sec:weakstat}, yielding
\begin{equation*}
  \int_{\Rn}\psi(t,\bx)\,\partial_t\rho\,d\bx
  = -E_V(t) - E_W(t) + E_D(t),
\end{equation*}
where $E_V(t), E_W(t), E_D(t)$ are the time-$t$ versions of the
integrals~\eqref{eq:EV_stat}--\eqref{eq:ED_stat} (with $f$ replaced by $\psi(t,\cdot)$).
Integrating both sides over $t\in[0,T]$ and applying integration by parts in time to
the left-hand side,
\begin{align*}
  \int_0^T\!\int_{\Rn}\psi\,\partial_t\rho\,d\bx\,dt
  &= \int_{\Rn}\psi(T,\bx)\rho(T,\bx)\,d\bx
   - \int_{\Rn}\psi(0,\bx)\rho_0(\bx)\,d\bx
   - \int_0^T\!\int_{\Rn}\partial_t\psi\cdot\rho\,d\bx\,dt,
\end{align*}
the space-time weak form becomes
\begin{equation}
  E_T - E_0 - E_t + E_V + E_W - E_D = 0,
  \label{eq:weaktime}
\end{equation}
where we use the plane-wave test functions
$\psi^{(k)}(t,\bx)=\sin(\bw^{(k)}\cdot\bx+\kappa^{(k)}t+b^{(k)})$ and define
\begin{align}
  E_T &= \int_{\Rn}\psi(T,\bx)\,\rho(T,\bx)\,d\bx,
    \label{eq:ET}\\
  E_0 &= \int_{\Rn}\psi(0,\bx)\,\rho_0(\bx)\,d\bx,
    \label{eq:E0}\\
  E_t &= \int_0^T\!\int_{\Rn}\frac{\partial\psi}{\partial t}(t,\bx)\,\rho(t,\bx)\,d\bx\,dt,
    \label{eq:Et}
\end{align}
\begin{align}
  E_V &= \int_0^T\!\int_{\Rn}\rho(t,\bx)\sum_{i=1}^n
    \frac{\partial\psi}{\partial x_i}\frac{\partial V}{\partial x_i}\,d\bx\,dt,
    \label{eq:EVt}\\
  E_W &= \int_0^T\!\int_{\mathbb{R}^{2n}}\rho(t,\bx)\,\rho(t,\by)\sum_{i=1}^n
    \frac{\partial\psi(t,\bx)}{\partial x_i}\frac{\partial W}{\partial x_i}(\bx-\by)\,d\by\,d\bx\,dt,
    \label{eq:EW}\\
  E_D &= \frac{\sigma^2}{2}\int_0^T\!\int_{\Rn}\rho(t,\bx)\sum_{i=1}^n
    \frac{\partial^2\psi}{\partial x_i^2}\,d\bx\,dt.
    \label{eq:EDt}
\end{align}
The derivatives of the plane-wave test function are
\begin{align}
  \frac{\partial\psi^{(k)}}{\partial t} &= \kappa^{(k)}\cos(\bw^{(k)}\cdot\bx+\kappa^{(k)}t+b^{(k)}),
    \nonumber\\
  \frac{\partial\psi^{(k)}}{\partial x_i} &= w_i^{(k)}\cos(\bw^{(k)}\cdot\bx+\kappa^{(k)}t+b^{(k)}),
    \nonumber\\
  \frac{\partial^2\psi^{(k)}}{\partial x_i^2} &= -(w_i^{(k)})^2\sin(\bw^{(k)}\cdot\bx+\kappa^{(k)}t+b^{(k)}).
  \label{eq:pwderiv_t}
\end{align}
The terms $E_T, E_0, E_t, E_V, E_D$ have exactly the same structure as in the standard
FPE weak form of \cite{HeCai2025}; the sole new ingredient is $E_W$, which involves a
double spatial integral over $\rho(t,\cdot)\otimes\rho(t,\cdot)$.

The adversarial loss is the squared residual of~\eqref{eq:weaktime} averaged over $K$
test functions:
\begin{equation}
  \mathcal{L}[\vth,\eta] = \frac{1}{K}\sum_{k=1}^K
  \Bigl(\hatE_T^{(k)} - \hatE_0^{(k)} - \hatE_t^{(k)}
    + \hatE_V^{(k)} + \hatE_W^{(k)} - \hatE_D^{(k)}\Bigr)^2,
  \label{eq:loss_time}
\end{equation}
where each $\hatE$ denotes the Monte Carlo estimator of the corresponding integral,
described below.

\subsection{Pushforward Representation}

We represent $\rho(t,\cdot)$ via a time-parameterized pushforward map:
\begin{equation}
  F_{\vth}(t,\bx_0,\br) = \bx_0 + \sqrt{t}\,\tilde{F}_{\vth}(t,\bx_0,\br),
  \label{eq:pf_time}
\end{equation}
where $\bx_0\sim\rho_0$, $\br\sim\pi_\mathrm{base}$, and
$\tilde{F}_{\vth}:\R^{1+n+d}\to\Rn$. At $t=0$, $F_{\vth}(0,\bx_0,\br)=\bx_0$,
enforcing the initial condition exactly.

\subsection{Monte Carlo Estimators}

For the standard terms $E_t, E_V, E_D$, we draw $M$ i.i.d.\ triples
$(t^{(m)},\bx_0^{(m)},\br^{(m)})$ with $t^{(m)}\sim\mathcal{U}(0,T)$,
$\bx_0^{(m)}\sim\rho_0$, $\br^{(m)}\sim\pi_\mathrm{base}$, and let
$\bxi^{(m)}=F_{\vth}(t^{(m)},\bx_0^{(m)},\br^{(m)})$. The terminal and initial terms
use dedicated batches of size $M_T$ and $M_0$ at fixed times $T$ and $0$.

\paragraph{Estimating $E_W$.}
The interaction term~\eqref{eq:EW} involves a double integral over $\rho(t,\cdot)$:
\begin{equation}
  E_W = \int_0^T \E_{\bxi\sim\rho_t}\!\bigl[\nabla_{\bxi}\psi(t,\bxi)\cdot
    \E_{\bfeta\sim\rho_t}[\nabla W(\bxi-\bfeta)]\bigr]\,dt.
  \label{eq:EW_rep}
\end{equation}
We introduce a dedicated batch of $M_W$ triples
$(t^{(m)},\bxi^{(m)},\bfeta^{(m)})$, where $t^{(m)}\sim\mathcal{U}(0,T)$ and
$\bxi^{(m)},\bfeta^{(m)}$ are two independent pushforward samples at time $t^{(m)}$.
The Monte Carlo estimator is
\begin{equation}
  \hatE_W^{(k)} = \frac{T}{M_W}\sum_{m=1}^{M_W}
    \nabla_{\bxi}\psi^{(k)}(t^{(m)},\bxi^{(m)})\cdot
    \nabla W(\bxi^{(m)}-\bfeta^{(m)}).
  \label{eq:EW_est}
\end{equation}
This estimator is unbiased for any $M_W\geq 1$.

\begin{remark}[Choice of $M_W$]
In practice one may set $M_W=M$ and reuse the time points $\{t^{(m)}\}$ drawn for
the interior terms, with $\bxi^{(m)}$ being the same pushed sample. However, this sample
reuse does not appreciably reduce training cost: $\bfeta^{(m)}$ still requires a full
forward and backward pass since it also depends on $\vth$. Setting $M_W$ independently
gives more flexibility to trade cost against estimator variance.
\end{remark}

\begin{remark}[Quadratic kernel simplification]
For $W(\bx-\by)=\frac{1}{2}\|\bx-\by\|^2$, we have $\nabla W(\bxi-\bfeta)=\bxi-\bfeta$.
Since $\E[\bfeta^{(m)}]=m(t^{(m)})$, one may replace $\bfeta^{(m)}$ by the batch mean
$\hatm(t^{(m)})$ to reduce variance, recovering the formulation of the stationary case
and eliminating the secondary sample entirely. This is what is done in our experiments.
\end{remark}

\begin{remark}[Singular kernels]
When $W$ is singular (e.g.\ $W(\bx)=-\log\|\bx\|$ for Keller--Segel), $\nabla W(\bxi-\bfeta)$
is large when $\bxi\approx\bfeta$. We adopt kernel mollification: replace $\nabla W(\bx)$
by $\nabla W_\varepsilon(\bx)$ evaluated at $\|\bx\|_\varepsilon:=\sqrt{\|\bx\|^2+\varepsilon^2}$
instead of $\|\bx\|$. The bias introduced is $O(\varepsilon^2)$ rather than $O(\varepsilon^{-(n-1)})$
because $\nabla W$ is odd ($\nabla W(-\bz)=-\nabla W(\bz)$) and $\bxi-\bfeta$ has a
symmetric distribution around zero, causing near-field contributions to cancel in pairs.
\end{remark}

\subsection{Tensor-Product Time Sampling for the Mean-Field Drift}

For the time-dependent McKean--Vlasov problem, accurate estimation of
$m(t)=\E_{\rho_t}[\bx]$ is essential because the drift depends on it. A naive approach
--- drawing all $M$ interior samples at $M$ different random times --- produces a
time-averaged constant $\hatm\approx\frac{1}{T}\int_0^T m(t)\,dt$, which biases the
drift and corrupts the learned mean evolution.

The correct approach uses a tensor-product sampling structure: draw $N_T$ fixed
quadrature time nodes $\{t_i\}_{i=1}^{N_T}$ on $(0,T]$ and, at each $t_i$, draw
$M_\mathrm{per}$ independent base samples $(\bx_0^{(j)},\br^{(j)})_{j=1}^{M_\mathrm{per}}$.
Then
\begin{equation}
  \hatm(t_i) = \frac{1}{M_\mathrm{per}}\sum_{j=1}^{M_\mathrm{per}}
    F_{\vth}\!\left(t_i,\bx_0^{(j)},\br^{(j)}\right)
  \label{eq:mean_est}
\end{equation}
is a valid unbiased estimate of $m(t_i)$, and the interior term is approximated by the
quadrature sum
\begin{equation}
  \hatE^{(k)} = \frac{T}{N_T}\sum_{i=1}^{N_T}\frac{1}{M_\mathrm{per}}
    \sum_{j=1}^{M_\mathrm{per}}
    \left[\frac{\partial\psi^{(k)}}{\partial t} + \Lop[\rho_{t_i}]\psi^{(k)}\right]
    \!\left(t_i, F_{\vth}(t_i,\bx_0^{(j)},\br^{(j)})\right).
  \label{eq:interior_est}
\end{equation}
The total number of forward passes $N_T\times M_\mathrm{per}$ matches the standard
WANPM interior batch size, so there is no additional computational cost.

Algorithm~\ref{alg:timedep} summarizes the training loop.

\begin{algorithm}[h]
\caption{WANPM for the McKean--Vlasov Equation (Time-Dependent)}
\label{alg:timedep}
\begin{algorithmic}[1]
\Require Pushforward $F_{\vth}$, test functions $\{{\psi}^{(k)}\}_{k=1}^K$,
batch sizes $M=N_T\times M_\mathrm{per}$, $M_0$, $M_T$, $M_W$,
learning rates $\eta_\mathrm{gen}$, $\eta_\mathrm{test}$.
\For{each training epoch}
  \State \textbf{Terminal:} sample $M_T$ pairs $(\bx_{0,T}^{(m)},\br_T^{(m)})$;
    compute $\hatE_T^{(k)}$.
  \State \textbf{Initial:} sample $M_0$ points $\bx_{0,0}^{(m)}\sim\rho_0$;
    compute $\hatE_0^{(k)}$.
  \State \textbf{Interior:} draw $N_T$ fixed time nodes $\{t_i\}$ and
    $M_\mathrm{per}$ base samples per node;
    compute $\hatm(t_i)$ via~\eqref{eq:mean_est} and
    $\hatE_t^{(k)},\hatE_V^{(k)},\hatE_D^{(k)}$ via~\eqref{eq:interior_est}.
  \State \textbf{Interaction ($E_W$):} sample $M_W$ pairs
    $(\bxi^{(m)},\bfeta^{(m)})$ at their shared times; compute $\hatE_W^{(k)}$ via~\eqref{eq:EW_est}.
  \State Assemble residuals $R^{(k)}=\hatE_T^{(k)}-\hatE_0^{(k)}-\hatE_t^{(k)}+\hatE_V^{(k)}+\hatE_W^{(k)}-\hatE_D^{(k)}$.
  \State Compute loss $\mathcal{L}=\frac{1}{K}\sum_k(R^{(k)})^2$.
  \State Generator step: $\vth\leftarrow\vth-\eta_\mathrm{gen}\nabla_{\vth}\mathcal{L}$.
  \If{adversary update step}
    \State $\eta\leftarrow\eta+\eta_\mathrm{test}\nabla_{\eta}\mathcal{L}$.
  \EndIf
\EndFor
\end{algorithmic}
\end{algorithm}

\subsection{Adversarial Frequency Initialization}
\label{sec:freqinit}

The initialization scale of the frequency vectors $\bw^{(k)}$ in the adversarial test
functions is critical for avoiding spurious minimizers.

For the quadratic kernel, the two-point distribution $\rho_\mathrm{two}$ has the same
mean and variance as the true stationary Gaussian, and satisfies the weak residual
approximately when test function frequencies are small. One can verify analytically
that the residual under $\rho_\mathrm{two}$ at a plane-wave test function with scalar
frequency $w$ and bias $b$ is
\begin{equation}
  R^{(k)}_\mathrm{two} = \sin(wa)\Bigl(\lambda a w\sin(b) - \tfrac{\sigma^2 w^2}{2}\cos(b)\Bigr),
  \label{eq:res_twopoint}
\end{equation}
while the exact Gaussian gives $R^{(k)}_\mathrm{Gaussian}=0$ for all $(w,b)$. For small
$w$ (specifically when $wa\ll 1$), the factor $\sin(wa)\approx wa$ is small, making
$|R^{(k)}_\mathrm{two}|$ comparably small to $|R^{(k)}_\mathrm{Gaussian}|$. In this
regime, the two-point distribution achieves a lower loss than the exact Gaussian.

When frequencies are initialized at a sufficiently large scale, the two-point distribution
incurs a large residual while the Gaussian remains near zero, so the adversary correctly
penalizes the spurious solution.

\paragraph{Dimension-dependent frequency scaling.}
Beyond this issue of spurious minimizers, there is a second and independent consideration:
the test function wavelength must match the spatial scale of the solution.

An $\mathcal{N}(0,\sigma_w^2 I_{d\times d})$ random vector has expected length
$\mathcal{O}(\sigma_w\sqrt{d})$, so the plane-wave test function
$\sin(\bw^{(k)}\cdot\bx+b^{(k)})$ has effective wavelength $\mathcal{O}(\sigma_w\sqrt{d}/\|\bx\|)$
near the bulk of the distribution. To correctly probe the solution at the spatial scale
$\mathcal{O}(\sigma/\sqrt{2\lambda})$, the frequency should satisfy
$\sigma_w\cdot(\sigma/\sqrt{2\lambda})\approx 1$, giving $\sigma_w\approx\sqrt{2\lambda}/\sigma$.
Equivalently, the test function inner product $\bw^{(k)}\cdot\bx$ should be
$\mathcal{O}(1)$ for typical samples $\bx$.

In higher dimensions, an additional correction arises: the dot product
$\bw^{(k)}\cdot\bx$ with $\bw^{(k)}\sim\mathcal{N}(0,\sigma_w^2 I)$ and $\bx$ at
the scale $\sigma/\sqrt{2\lambda}$ has standard deviation
$\sigma_w\cdot(\sigma/\sqrt{2\lambda})\cdot\sqrt{n}$. Requiring this to be
$\mathcal{O}(1)$ gives the scaling
\begin{equation}
  \sigma_w \approx \sqrt{2\lambda}/(\sigma\sqrt{n}).
  \label{eq:freqscaling}
\end{equation}
This is consistent with the empirically observed rule: for $d=2$,
$\sigma_w\sim\mathcal{N}(0,2^2)$ works well, while for $d=20$, initializing
$\sigma_w\sim\mathcal{N}(0,0.2^2)$ is needed.
Table~\ref{tab:freqscaling} summarizes the initialization scales used in our experiments.

\begin{table}[h]
\centering
\caption{Frequency initialization scales used in experiments. All experiments use
$\theta=1$, $\sigma=1$, $\lambda=2$. The prediction~\eqref{eq:freqscaling} gives
$\sigma_w\approx 2/\sqrt{n}$.}
\label{tab:freqscaling}
\begin{tabular}{cccc}
\toprule
Dimension $n$ & $\sigma_w$ (used) & $2/\sqrt{n}$ (predicted) & Type \\
\midrule
2  & 2.0 & 1.41 & Stationary \\
20 & 0.3 & 0.45 & Stationary \\
5  & 0.1 & 0.89 & Transient \\
100 & 0.1 & 0.20 & Transient \\
\bottomrule
\end{tabular}
\end{table}

\section{Numerical Experiments}
\label{sec:experiments}

\subsection{Benchmark Problem: Linear McKean--Vlasov}

Throughout, we use $V(\bx)=\frac{\theta}{2}\|\bx\|^2$ and
$W(\bx-\by)=\frac{1}{2}\|\bx-\by\|^2$ with $\theta=1$, $\sigma=1$, giving $\lambda=2$.
The exact solution is Gaussian at all times with moments satisfying
\begin{equation}
  \dot{m}(t) = -\theta\,m(t), \qquad
  \dot{\Sigma}(t) = -2(\theta+1)\Sigma(t)+\sigma^2,
  \label{eq:ODEs}
\end{equation}
whose solutions are
\begin{equation}
  m(t) = m_0\,e^{-\theta t}, \qquad
  \Sigma(t) = \left(\Sigma_0 - \frac{\sigma^2}{2\lambda}\right)e^{-2\lambda t}
    + \frac{\sigma^2}{2\lambda}.
  \label{eq:exact}
\end{equation}
The stationary solution is $\rho^*=\mathcal{N}(0,\frac{\sigma^2}{2\lambda}I)
=\mathcal{N}(0,0.25\,I)$ with standard deviation $0.5$ per component.

All experiments use the quadratic kernel simplification of Section~\ref{sec:quadratic}
(batch mean trick for the mean-field term), the adversarial optimizer SGD with
$\eta_\mathrm{test}=10^{-2}$, the generator optimizer Adam with
$\eta_\mathrm{gen}=10^{-3}$, and one adversary update every two generator steps. The
base distribution is $\mathcal{U}[0,1)^{D_\mathrm{base}}$ for stationary experiments
and $\mathcal{N}(0,I)$ for transient experiments.

\subsection{Experiment 1: 2D Stationary McKean--Vlasov}
\label{sec:exp3}

We begin with the two-dimensional stationary problem, whose exact solution
$\rho^*=\mathcal{N}(0,0.25\,I_2)$ provides a stringent test of the method's ability to
produce the correct shape and not merely the correct first two moments. The pushforward
network maps a base noise vector of dimension $D_\mathrm{base}=8$ through three hidden
layers of width $128$ with Tanh activations to $\mathbb{R}^2$, and is trained against an
ensemble of $K=2000$ plane-wave test functions whose frequencies are initialized at scale
$2.0$. Both the primary batch and the secondary sample used to estimate $E_W$ contain
$M=2000$ and $M_W=4000$ samples respectively. Training runs for $5000$ epochs with an
adversary update every two generator steps, and completes in $27.4\,$s on a single GPU.

The loss converges to $1.15\times 10^{-3}$, and the learned per-component standard
deviations are $0.492$ and $0.474$, against the exact value of $0.500$. As shown in
Figure~\ref{fig:exp3}, the scatter of learned samples closely hugs the exact $2\sigma$
ellipse, and the marginal histogram of $x_1$ lies nearly on top of the exact
$\mathcal{N}(0,0.25)$ density throughout its support.

\begin{figure}[h]
\centering
\includegraphics[width=0.92\textwidth]{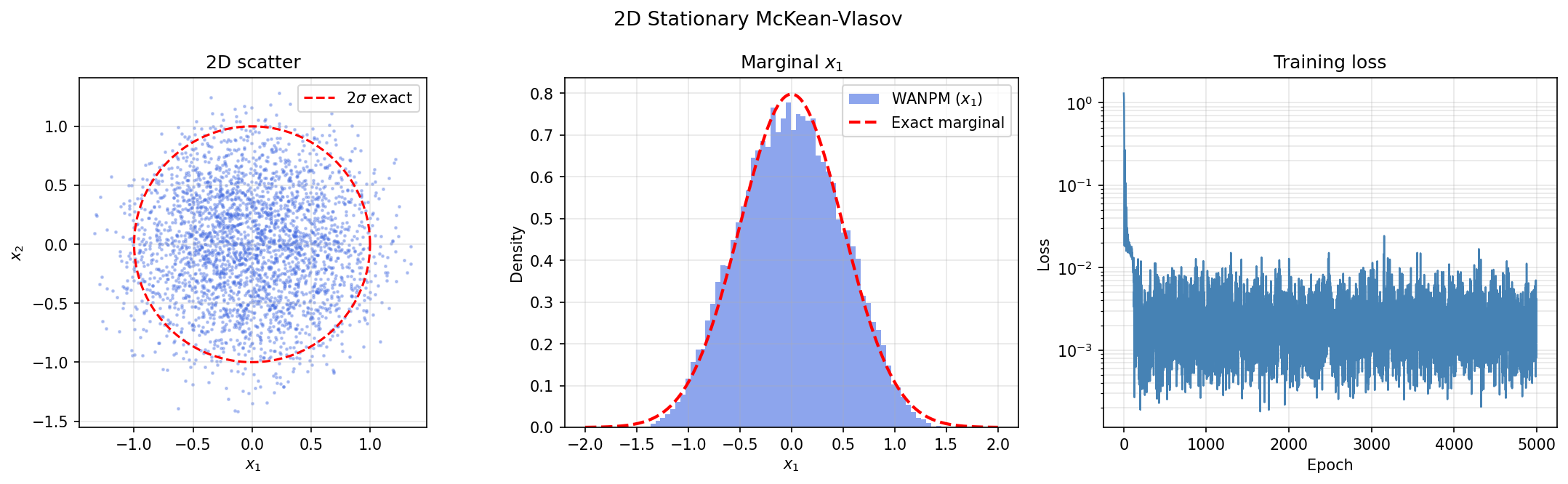}
\caption{Experiment 1 (2D Stationary). \emph{Left:} scatter plot of $3000$ learned
samples with the exact $2\sigma$ ellipse (red dashed). \emph{Center:} marginal density
of $x_1$ (histogram) vs.\ exact $\mathcal{N}(0,0.25)$ (red dashed).
\emph{Right:} training loss convergence over $5000$ epochs.}
\label{fig:exp3}
\end{figure}

\subsection{Experiment 2: 20D Stationary McKean--Vlasov}
\label{sec:exp5}

To examine scalability in the stationary setting, we solve the same problem in $n=20$
dimensions, where the exact solution is $\rho^*=\mathcal{N}(0,0.25\,I_{20})$. The
pushforward network is deliberately kept compact --- a two-hidden-layer architecture of
width $64$ with $D_\mathrm{base}=30$ and $7{,}444$ trainable parameters total --- to
demonstrate that the method does not require a large model to handle moderate dimensions.
Following the frequency-scaling rule~\eqref{eq:freqscaling}, the $K=5000$ test function
frequencies are initialized at scale $0.3$, substantially smaller than in the 2D case.
The training batch is $M=5000$ with $M_W=10000$ for the interaction term, and training
runs for $10{,}000$ epochs in $134.9\,$s.

The final loss is $3.77\times 10^{-4}$. Aggregated over all 20 dimensions, the mean
absolute error in the per-component mean is $0.0123$ and in the per-component variance
is $0.00406$, both well below $1\%$ of the exact values. Figure~\ref{fig:exp5a} confirms
that all 20 standard deviations are recovered accurately, with no systematic bias across
dimensions. Selected marginal densities, shown in Figure~\ref{fig:exp5b}, agree closely
with the exact $\mathcal{N}(0,0.25)$ density.

\begin{figure}[h]
\centering
\includegraphics[width=0.9\textwidth]{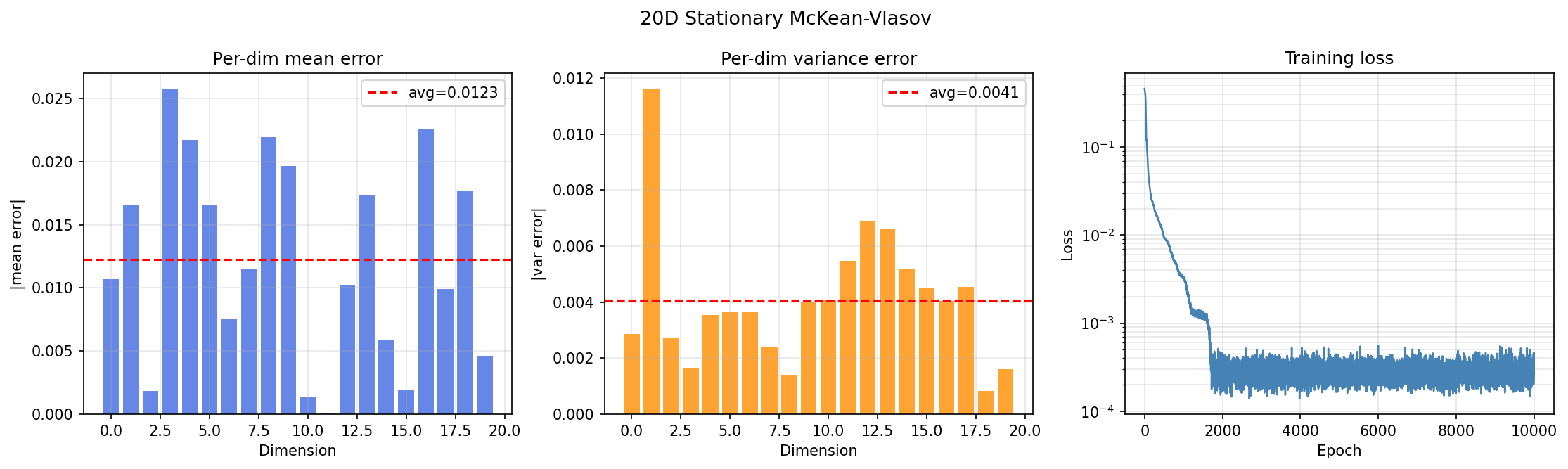}
\caption{Experiment 2 (20D Stationary). Per-component mean and standard deviation of
$20000$ learned samples, compared to exact values (dashed red). All 20 dimensions
are accurately recovered.}
\label{fig:exp5a}
\end{figure}

\begin{figure}[h]
\centering
\includegraphics[width=0.9\textwidth]{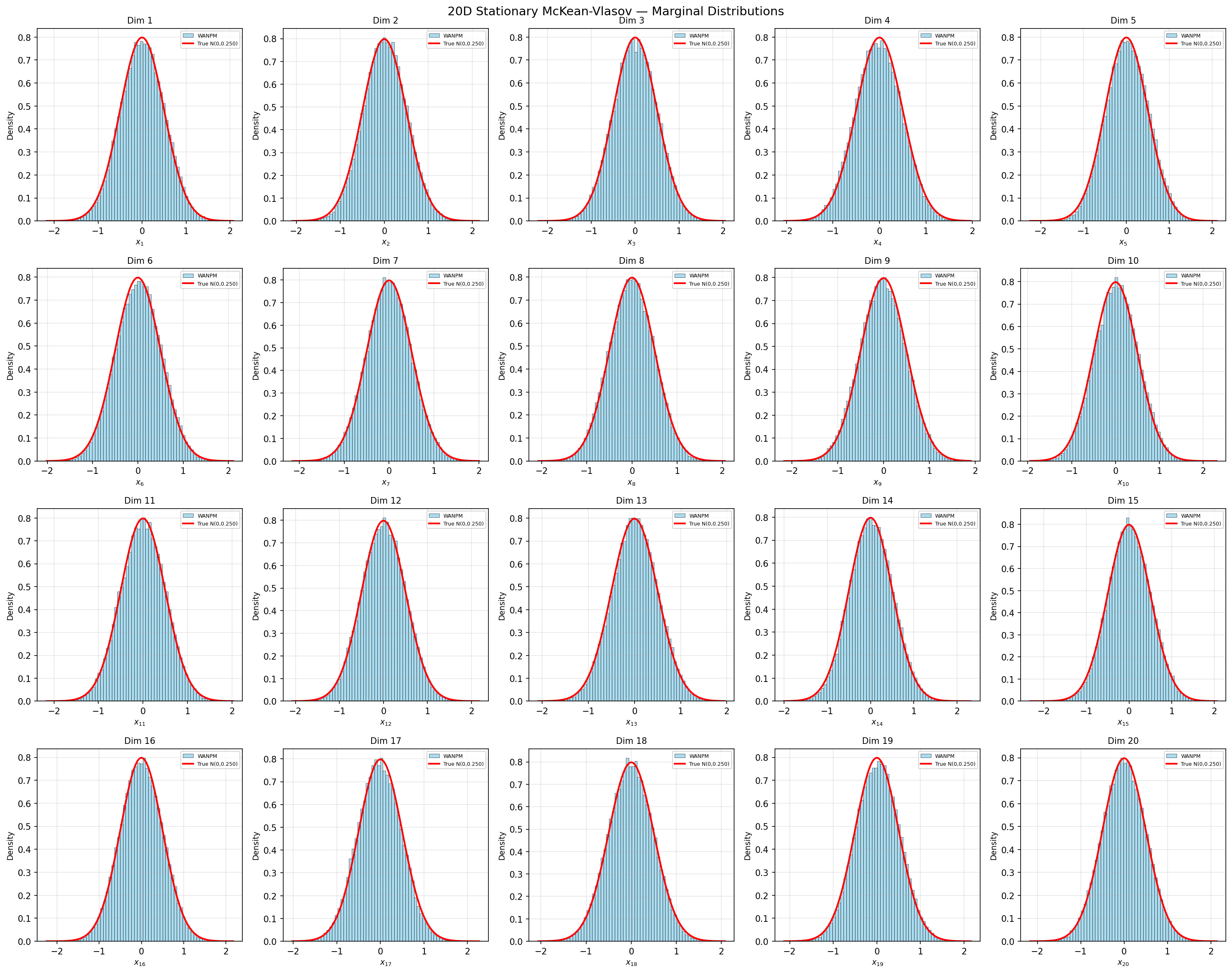}
\caption{Experiment 2 (20D Stationary). Selected marginal densities (histogram) vs.\
exact $\mathcal{N}(0,0.25)$ (red dashed).}
\label{fig:exp5b}
\end{figure}

\subsection{Experiment 3: 5D Transient McKean--Vlasov}
\label{sec:exp4}

We now turn to the time-dependent problem in $n=5$ dimensions over the interval
$[0,T]=[0,1]$. The initial distribution is $\rho_0=\mathcal{N}(\mu_0,0.25\,I_5)$, where
$\mu_0=(3.528,0.800,1.957,\ldots)$ is a random vector with components of order $2$,
chosen to make the transient mean evolution clearly visible before it decays to zero.
The exact solution is the Gaussian $\mathcal{N}(m(t),\Sigma(t)I_5)$ with moments
given by~\eqref{eq:exact}.

The pushforward takes the form $F_{\vth}(t,\bx_0,\br)=\bx_0+\sqrt{t}\,\tilde{F}_{\vth}(t,\br)$
with a base noise dimension of $D_\mathrm{base}=16$ and a three-hidden-layer network of
width $128$. A total of $K=3000$ test functions are used, initialized at frequency scale
$0.1$ as dictated by the high-dimensional regime. The interior batch size is $M=3000$
(split as $N_T$ time nodes with $M_\mathrm{per}$ samples each), the interaction batch is
$M_W=6000$, and the terminal and initial batches are $M_0=M_T=1000$. Training
for $10{,}000$ epochs takes $110.3\,$s.

The final loss is $1.96\times10^{-5}$. The per-component mean and variance errors at
three representative times are tabulated below; the variance error is particularly small
at $t=1$, consistent with the distribution having nearly relaxed to its equilibrium
$\mathcal{N}(0,0.25\,I_5)$ by that time.

\begin{center}
\begin{tabular}{ccc}
\toprule
$t$ & $|\bar{m}_\mathrm{err}|$ (avg) & $|\bar{\Sigma}_\mathrm{err}|$ (avg) \\
\midrule
$0.1$ & $0.0237$ & $0.0452$ \\
$0.5$ & $0.0643$ & $0.0115$ \\
$1.0$ & $0.0267$ & $0.0023$ \\
\bottomrule
\end{tabular}
\end{center}

Figure~\ref{fig:exp4a} shows the learned mean and variance trajectories for all five
components overlaid on the exact curves; the method tracks both quantities faithfully
throughout the transient. Marginal histograms at $t=T=1$ are shown in
Figure~\ref{fig:exp4b}.

\begin{figure}[h]
\centering
\includegraphics[width=0.95\textwidth]{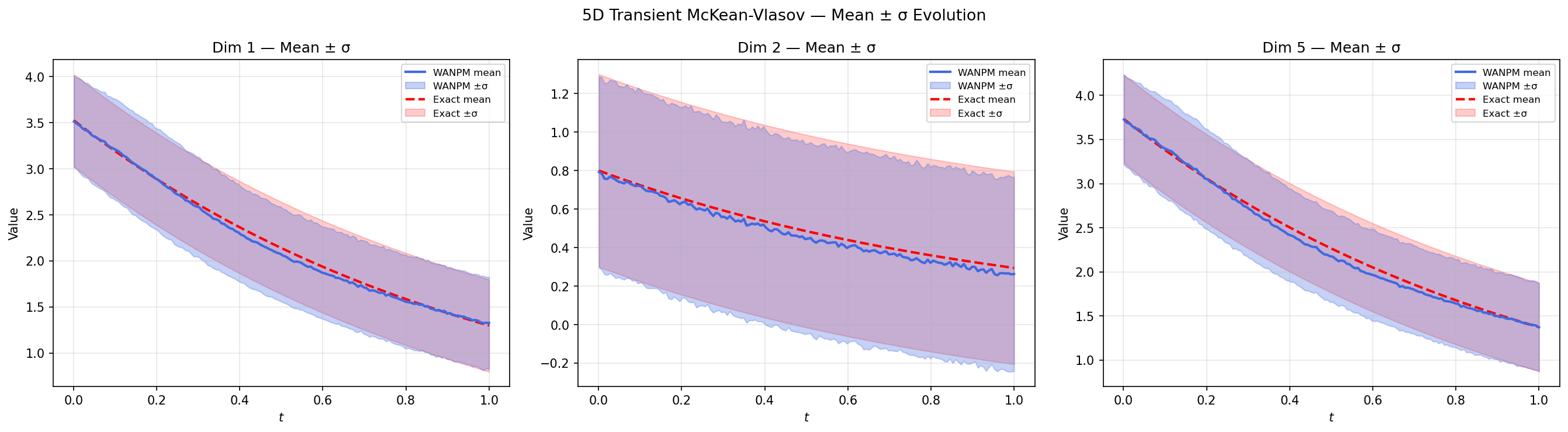}
\caption{Experiment 3 (5D Transient). Mean and variance evolution over $[0,1]$ for all
5 components: learned (solid) vs.\ exact (dashed). The mean decays as
$m_i(t)=m_{0,i}e^{-\theta t}$ and the variance relaxes to the equilibrium value $0.25$.}
\label{fig:exp4a}
\end{figure}

\begin{figure}[h]
\centering
\includegraphics[width=0.95\textwidth]{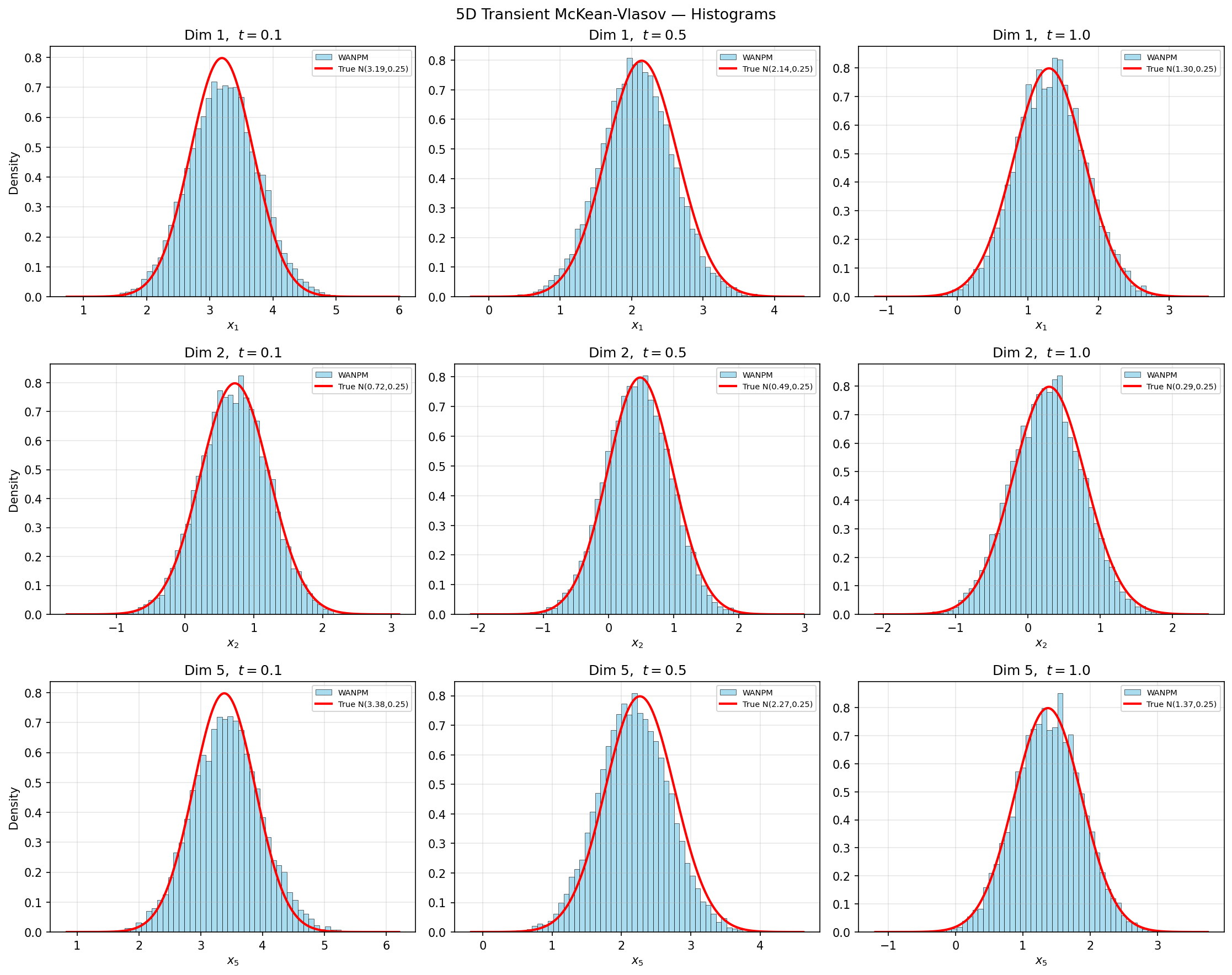}
\caption{Experiment 3 (5D Transient). Per-component marginal histograms at $t=T=1$
vs.\ exact Gaussian (red dashed).}
\label{fig:exp4b}
\end{figure}

\subsection{Experiment 4: 100D Transient McKean--Vlasov}
\label{sec:exp6}

The final experiment pushes the method to $n=100$ dimensions, where the exact solution
is again a time-evolving Gaussian with initial mean $\mu_0$ sampled from
$\mathcal{N}(0,I_{100})$ and initial variance $0.25$ per component. The scale of the
problem necessitates a larger network: $D_\mathrm{base}=200$ with three hidden layers of
width $256$, giving $71{,}780$ parameters, run on an NVIDIA A100-SXM4-80GB GPU. The
interior and interaction batch sizes are $M=10{,}000$ and $M_W=20{,}000$; terminal and
initial batches each contain $M_0=M_T=2{,}000$ samples. The $K=5{,}000$ test functions
are initialized at frequency scale $0.1$, in line with the scaling rule for high
dimensions. Training for $10{,}000$ epochs completes in $10.3\,$minutes.

Despite the formidable dimension, the method achieves a final loss of $2.06\times10^{-4}$.
The per-component errors, reported in the table below, show that the average absolute
mean error stays below $2.6\%$ across all times, and the average variance error remains
below $0.5\%$ of the equilibrium variance throughout the evolution.

\begin{center}
\begin{tabular}{ccccc}
\toprule
$t$ & $|\bar{m}_\mathrm{err}|$ (avg) & $|\bar{\Sigma}_\mathrm{err}|$ (avg)
    & $\|m_\mathrm{err}\|_2$ & $\|\Sigma_\mathrm{err}\|_2$ \\
\midrule
$0.1$ & $0.0259$ & $0.00467$ & $0.329$ & $0.0674$ \\
$0.5$ & $0.0156$ & $0.00443$ & $0.196$ & $0.0644$ \\
$1.0$ & $0.0103$ & $0.00335$ & $0.130$ & $0.0418$ \\
\bottomrule
\end{tabular}
\end{center}

The $\ell^2$ norms of the errors decrease monotonically in time, reflecting the
contraction of the distribution toward its stationary state. Figure~\ref{fig:exp6loss}
shows the loss trajectory over $10{,}000$ epochs. Figure~\ref{fig:exp6evol} displays the
learned mean and standard deviation curves for all $100$ components simultaneously,
confirming that the method tracks the $e^{-\theta t}$ mean decay and the variance
relaxation without confusion between components despite the high dimensionality.
Representative marginal histograms at $t=1$ in Figure~\ref{fig:exp6hist} agree closely
with the exact Gaussian, and similar agreement is observed across all components.

\begin{figure}[h]
\centering
\includegraphics[width=0.65\textwidth]{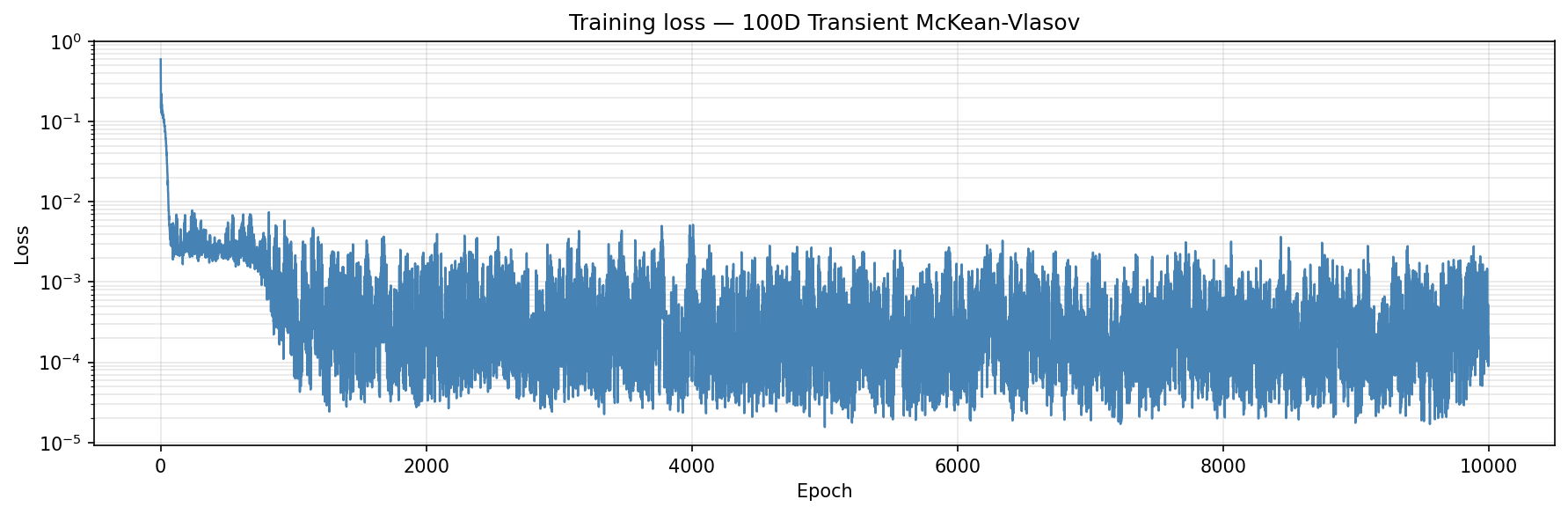}
\caption{Experiment 4 (100D Transient). Training loss over $10000$ epochs on an A100
GPU. Total training time: $10.3\,$min.}
\label{fig:exp6loss}
\end{figure}

\begin{figure}[h]
\centering
\includegraphics[width=0.95\textwidth]{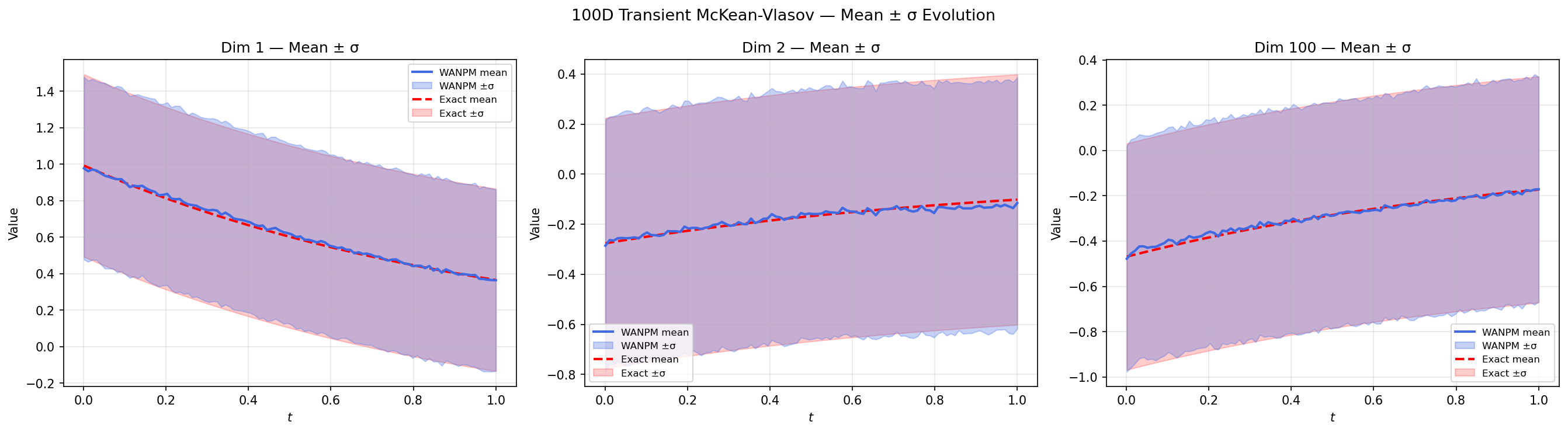}
\caption{Experiment 4 (100D Transient). Mean (top) and standard deviation (bottom) over
time for all 100 components: learned (blue) vs.\ exact (red dashed). The mean decays
at rate $e^{-\theta t}$ and the variance relaxes to $0.25$.}
\label{fig:exp6evol}
\end{figure}

\begin{figure}[h]
\centering
\includegraphics[width=0.95\textwidth]{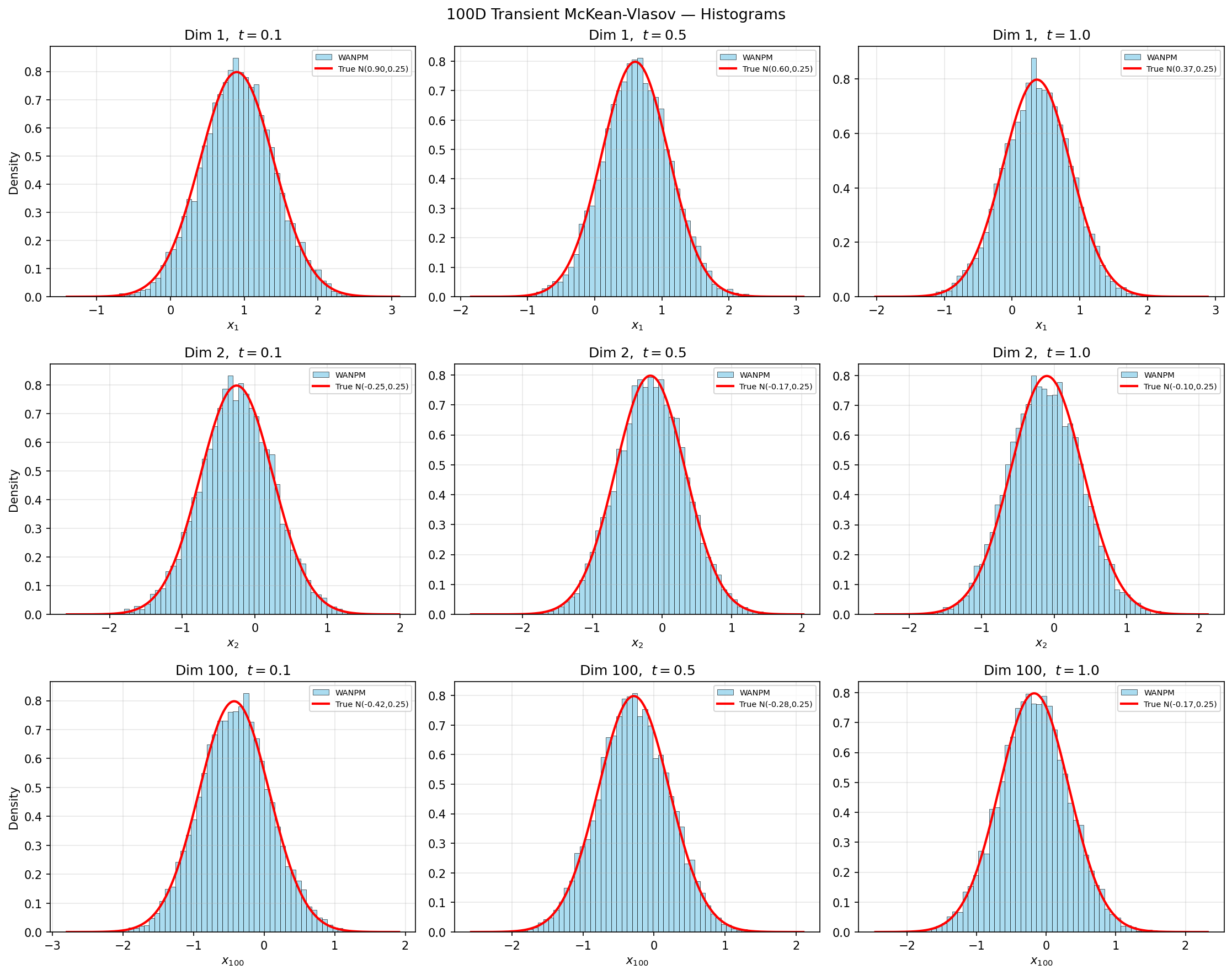}
\caption{Experiment 4 (100D Transient). Selected marginal histograms at $t=1$ vs.\
exact Gaussian (red dashed). Results are representative of all 100 components.}
\label{fig:exp6hist}
\end{figure}

\section{Conclusion}
\label{sec:conclusion}

We have developed the Weak Adversarial Neural Pushforward Method for the McKean--Vlasov
mean-field Fokker--Planck equation, covering both the stationary and time-dependent
settings. The central findings are as follows.

For the quadratic (granular media) interaction kernel, the mean-field nonlinearity
reduces to the sample mean of the primary training batch, requiring no secondary sampling
and no modification to the pushforward architecture. The same principle extends to
polynomial-degree and separable kernels, for which the interaction integrals reduce to
batch moments or scalar expectations.

The initialization scale of adversarial test function frequencies is critical. Small
initial frequencies allow the two-point distribution, which matches the exact Gaussian in
mean and variance, to achieve a lower loss than the true solution. The required scale
follows the dimension-dependent rule~\eqref{eq:freqscaling}: for an $n$-dimensional
problem the scale should decrease as $1/\sqrt{n}$. The adversarial training is in
principle self-adaptive, but does not always converge to the correct frequency without a
good initialization.

For the time-dependent problem, a tensor-product sampling structure over time and space
is needed to produce valid per-time mean-field estimates without bias. Fixed quadrature
nodes in time, with independent Monte Carlo draws at each node, resolve this correctly at
no extra computational cost.

Numerical experiments on the linear McKean--Vlasov benchmark in 2, 5, 20, and 100
dimensions confirm accurate recovery of both the stationary and transient Gaussian
distributions, with per-component mean and variance errors at the $1$--$5\%$ level and
training times from $27\,$s to $10\,$min on a single GPU.

Natural extensions include higher-dimensional stationary problems, separable and singular
interaction kernels (requiring mollification of $\nabla W$ for the Keller--Segel case),
and granular-media-type kernels exhibiting phase transitions. The Keller--Segel kernel
$W(\bx)=-\log\|\bx\|$ requires a secondary batch for the interaction integral, and the
singular-kernel estimator relies on the oddness cancellation analyzed in
Remark~3 of Section~\ref{sec:timedep}.

\end{document}